\documentclass[11pt]{article}

\usepackage{fullpage}
\usepackage{amsmath, amsthm, amsfonts, amssymb, amstext, mathrsfs, enumerate}
\usepackage{graphicx, ragged2e, lscape, framed, xcolor}
\usepackage{subfiles}
\usepackage[T1]{fontenc}

\theoremstyle{plain}
\newtheorem{theorem}{Theorem}[section]
\newtheorem{lemma}[theorem]{Lemma}

\newtheorem{conjecture}[theorem]{Conjecture}
\newtheorem{corollary}[theorem]{Corollary}
\newtheorem{problem}[theorem]{Problem}

\numberwithin{equation}{section}
\allowdisplaybreaks

\newcommand{\affl}[3]{\noindent #1, Email: {\tt #2}\\ \textsc{#3}\\[1.5pt]}

\usepackage[pagebackref]{hyperref}
\hypersetup{
	colorlinks=true,
    urlcolor=purple,
	linkcolor=purple,
    citecolor=purple,
}

\DeclareMathOperator{\diam}{diam}
\DeclareMathOperator{\PG}{PG}

\title{\textbf{An improved bound for the strong clique index of graphs}}
\author{Hitesh Kumar \and Bojan Mohar \and Shivaramakrishna Pragada}
\date{}

\begin{document}
\maketitle
\begin{abstract} For a graph $G$ with line graph $L(G)$, $\chi(L(G)^2)$ and $\omega(L(G)^2)$ are called the \emph{strong chromatic index} and \emph{strong clique index} of $G$, respectively. A well-known conjecture of Erd\H{o}s and Ne\v{s}et\v{r}il (1985) posits that $\chi(L(G)^2)\le \frac{5}{4}\Delta(G)^2$. Related to that, 
 Faudree, Gy\'{a}rf\'{a}s, Schelp and Tuza (1990) conjectured that $\omega(L(G)^2) \le \frac{5}{4}\Delta(G)^2$. We show that $\omega(L(G)^2) \le \frac{2607}{1987}\Delta(G)^2 < \frac{21}{16}\Delta(G)^2$ improving the upper bound $\frac{4}{3}\Delta(G)^2$ of Faron and Postle. Indeed, we make progress towards a stronger conjecture of Faron and Postle in terms of Ore-degree. 

 For positive integers $\Delta$ and $t$, let $h_t(\Delta)$ denote the smallest integer such that any graph $G$ with size at least $h_t(\Delta)$ and maximum degree $\Delta(G)\le \Delta$, contains two edges with distance at least $t$. An old problem of Erd\H{o}s and Ne\v{s}et\v{r}il (1986) concerns estimating the quantity $h_t(\Delta)$ and can be thought of as the edge-version of the degree-diameter problem. Chung, Gy\'{a}rf\'{a}s, Tuza and Trotter established the sharp inequality $h_2(\Delta)\le \frac{5}{4}\Delta^2+1$. We disprove two conjectures of Cambie, Cames van Batenburg, Joannis de Verclos and Kang concerning the next open case $h_3(\Delta)$. 
\end{abstract}



\section{Introduction}

For a finite simple graph $G$ and a positive integer $k$, the \emph{$k$-th power} of $G$, denoted by $G^k$, is the graph with vertex set $V(G)$ and two vertices are adjacent in $G^k$ if and only if they are at distance at most $k$ in $G$. The graph $G^2$ is often referred to as the \emph{square} of $G$. We denote the \emph{line graph} of $G$ by $L(G)$. The quantities $\chi(L(G)^2)$ and $\omega(L(G)^2)$ are called the \emph{strong chromatic index} and \emph{strong clique index} of $G$, respectively. 

\subsection{Strong clique index}

There is a vast literature on colorings of graph powers, particularly graph squares. We refer the reader to the excellent survey by Cranston \cite{Cranston_2023}. A well-known conjecture in this area concerns the chromatic number of squares of line graphs, a.k.a. the strong chromatic index of $G$. In the 1980s, Erd\H{o}s and Ne\v{s}et\v{r}il (see \cite{Halasz_Sos_1989, Faudree_Gyarfas_Schelp_Tuza_1989}) proposed the following conjecture relating strong chromatic index and maximum degree of graphs.

\begin{conjecture}[Erd\H{o}s--Ne\v{s}et\v{r}il \cite{Halasz_Sos_1989, Faudree_Gyarfas_Schelp_Tuza_1989}]\label{conj:chi_line_square} For any graph $G$, 
\[ \chi(L(G)^2) \le \frac{5}{4}\Delta(G)^2.\]
\end{conjecture}

Apparently, inspired by the difficulty of this conjecture, the following weaker conjecture was made by Faudree, Gy\'{a}rf\'{a}s, Schelp, and Tuza \cite{Faudree_Gyarfas_Schelp_Tuza_1990} concerning the strong clique index. 

\begin{conjecture}[Faudree--Gy\'{a}rf\'{a}s--Schelp--Tuza \cite{Faudree_Gyarfas_Schelp_Tuza_1990}]\label{conj:omega_line_square}
    For any graph $G$, 
    \[\omega(L(G)^2) \le \frac{5}{4}\Delta(G)^2.\]
\end{conjecture}

Both of the above conjectures are tight for the blowup $C_5^{(t)}$ of cycle $C_5$, i.e., the graph obtained from $C_5$ by replacing each vertex with an independent set of size $t$ and replacing edges with complete bipartite graphs. One can check that $\Delta(C_5^{(t)}) = 2t$ and $L(C_5^{(t)})^2$ is a complete graph of order $5t^2 = \frac{5}{4}\Delta(C_5^{(t)})^2$.

\'{S}leszy\'{n}ska-Nowak \cite{Nowak_2016} proved the general upper bound $\frac{3}{2}\Delta(G)^2$ for $\omega(L(G)^2)$ which significantly improves on the trivial upper bound $2\Delta(G)^2$. Soon after, Faron and Postle \cite{Faron_Postle_2019} established the following bound, which is the best-known general upper bound to date.

\begin{theorem}[\cite{Faron_Postle_2019}]\label{thm:Faron_Postle_bound} For any graph $G$, 
\[\omega(L(G)^2) \le \frac{4}{3}\Delta(G)^2.\]      
\end{theorem}

For further partial progress on Conjecture \ref{conj:omega_line_square}, we refer the reader to \cite{Debski_Nowak_2021, Debski_Nowak_2022, Cho_Choi_Kim_Park_2021, Batenburg_Kang_Pirot_2020} and the survey \cite{Cranston_2023}. 

We quickly summarize the approach of Faron and Postle \cite{Faron_Postle_2019}. They suggested that it is more helpful to work with the Ore-degree instead of the maximum degree. Recall that if $H$ is a non-empty subgraph of $G$, the \emph{Ore-degree} of $H$ in $G$ is defined to be
\[
   \sigma_G(H):=\max_{xy\in E(H)}(\deg_G(x)+\deg_G(y)).
\]
Let $\sigma_G(H) = 0$ if $H$ is an empty graph. They crucially identified the following stronger conjecture, which implies Conjecture \ref{conj:omega_line_square}.

\begin{conjecture}[\cite{Faron_Postle_2019}]\label{conj:bipartite_subgraph} Let $H$ be a bipartite subgraph of $G$ such that $E(H)$ forms a clique in $L(G)^2$. Then 
\[ |E(H)|\le \frac{1}{4} \ \sigma_G(H)^2.\]
\end{conjecture}

Indeed, they proved the following.

\begin{theorem}[\cite{Faron_Postle_2019}]\label{thm:conditional_result}
Let $H$ be a subgraph of $G$ such that $E(H)$ is a clique in $L(G)^2$. Let $\beta \in [\frac{1}{4}, \frac{1}{3}]$ be such that the following assumption holds: for any bipartite subgraph $H'$ of $H$ such that $|E(H')|<|E(H)|$, \[|E(H')|\le \beta \ \sigma_{G[V(H')]}(H')^2.\]
Then
\[ |E(H)|\le \left(\frac{1 + \beta}{4}\right)\sigma_G(H)^2. \]
\end{theorem}

Clearly, if Conjecture \ref{conj:bipartite_subgraph} holds with $\beta = \frac{1}{4}$, then the assumption in Theorem \ref{thm:conditional_result} holds with $\beta = \frac{1}{4}$ for the subgraph $H$ of $G$ whose edges induce the largest clique in $L(G)^2$, which in turn implies Conjecture \ref{conj:omega_line_square}. Faron and Postle observed that Conjecture \ref{conj:bipartite_subgraph} holds for $\beta = \frac{1}{3}$ inductively using Theorem \ref{thm:conditional_result}, thus establishing Theorem \ref{thm:Faron_Postle_bound}. 

We make non-trivial progress towards Conjecture \ref{conj:bipartite_subgraph} and show the following.

\begin{theorem}\label{thm:bipartite_subgraph}
Let $G$ be a graph and let $H$ be a bipartite subgraph of $G$. If
$E(H)$ is a clique in $L(G)^2$, then
\[
   |E(H)|\le \frac{620}{1987}\,\sigma_G(H)^2.
\]
\end{theorem}

Applying Theorem \ref{thm:conditional_result} with $\beta =\frac{620}{1987}$, we get the following improved bound for $\omega(L(G)^2)$. 

\begin{corollary}\label{cor:omega_line_square}
For every graph $G$,
\[
   \omega(L(G)^2)
   \le \frac{2607}{1987}\,\Delta(G)^2.
\]
\end{corollary}

To compare these results with Theorem \ref{thm:conditional_result} and Conjecture \ref{conj:bipartite_subgraph}, observe that $\frac{620}{1987} < \frac{5}{16}$ and $\frac{2607}{1987}<\frac{21}{16}$. There is room for minor optimization in our proofs, which we avoid to keep the presentation short and clean. We believe new ideas are needed to bring down the coefficient below $1.3$ in Corollary \ref{cor:omega_line_square}. We prove Theorem \ref{thm:bipartite_subgraph} and Corollary \ref{cor:omega_line_square} in Section \ref{section:clique_number_line_square}.

\subsection{Degree-diameter problem for the edges}

For given positive integers $\Delta$ and $t$, define
\begin{align}\label{eq:h_t_definition}
     h_t(\Delta) - 1& :=\max_G \ \{ \ |E(G)|: \Delta(G)\le \Delta, \ L(G)^t \text{ is a complete graph }\}\nonumber\\
     & \le \max_G \{ \ \omega(L(G)^t): \Delta(G)\le \Delta \ \}.
\end{align}

Equivalently, $h_t(\Delta)$ is the smallest integer such that any graph $G$ with size at least $h_t(\Delta)$, maximum degree $\Delta(G)\le \Delta$, contains two edges with distance at least $t$ in $G$ (i.e., there are two vertices in the line graph $L(G)$ at distance at least $t+1$). An old problem of Erd\H{o}s and Ne\v{s}et\v{r}il \cite{Erdos_1986} concerns estimating the quantity $h_t(\Delta)$ which can be thought of as the edge-version of the well-known degree-diameter problem, see the survey \cite{Miller_Siran_2005}. 

If $t=1$, then it is clear that $h_1(\Delta) = \Delta + 1$. For $t = 2$, Erd\H{o}s and Ne\v{s}et\v{r}il \cite{Erdos_1986} and Bermond, Bond, Paoli and Peyrat \cite{Bermond_Bond_Paoli_Peyrat_1983} independently conjectured that $h_2(\Delta) \le \frac{5}{4}\Delta^2 + 1$. This was later established by Chung, Gy\'{a}rf\'{a}s, Tuza and Trotter \cite{Chung_Gyarfas_Tuza_Trotter_1990}. The open blow-ups of $C_5$ give tight examples. Cambie, Cames van Batenburg, Joannis de Verclos and Kang \cite{Cambie_Batenburg_Verclos_Kang_2022}, proved the following general upper bound for $\omega(L(G)^t)$, which by \eqref{eq:h_t_definition} gives the best known upper bound on $h_t(\Delta)$ for $t\ge 3$. 

\begin{theorem}[\cite{Cambie_Batenburg_Verclos_Kang_2022}] For any graph $G$, we have 
\[\omega(L(G)^t) \le \frac{3}{2}\Delta^t.\]
\end{theorem}

In \cite{Cambie_Batenburg_Verclos_Kang_2022}, the authors further proposed the following conjectures. 

\begin{conjecture}[\cite{Cambie_Batenburg_Verclos_Kang_2022}]\label{conj:power_3} $h_3(\Delta)\le \Delta^3-\Delta^2+\Delta+2$.
\end{conjecture}

\begin{conjecture}[\cite{Cambie_Batenburg_Verclos_Kang_2022}]\label{conj:large_max_degree} For $t\ge 3$ and every
    $\varepsilon>0$,
    \[
        h_t(\Delta)\le (1+\varepsilon)\Delta^t
    \]
    for all sufficiently large $\Delta$.
\end{conjecture}

In \cite{Cambie_Batenburg_Verclos_Kang_2022}, Conjecture \ref{conj:power_3} was verified for $\Delta = 3$. Here, we disprove both of these conjectures. We first show that the $4$-regular Odd graph $O_4$ and the $15$-regular truncated Witt graph $W$ are counterexamples to Conjecture \ref{conj:power_3}. Then, using projective planes $\PG(2,q)$, we construct an infinite family of graphs $G[H,q]$ (see Lemma \ref{lemma:G[H,q]}). Taking $H$ to be $O_4$ or $W$ and letting $q\to \infty$ disproves Conjecture \ref{conj:large_max_degree} when $t = 3$ and Conjecture \ref{conj:power_3} for all sufficiently large $\Delta$. Indeed, we show the following.

\begin{theorem}\label{thm:infinite_counter_power_3}
 We have
\[
\liminf_{\Delta\to\infty}\frac{h_3(\Delta)}{\Delta^3}\ge \frac{253}{225}.
\]
Equivalently, for every $0<\varepsilon<28/225$, and sufficiently large $\Delta$, we have 
\[
        h_3(\Delta)>(1+\varepsilon)\Delta^3.
\]
\end{theorem}

We prove Theorem \ref{thm:infinite_counter_power_3} in Section \ref{section:counterexample_power_3}. We remark here that Conjecture \ref{conj:large_max_degree} remains undecided for $t\ge 4$. We propose the following problem for $t=3$. 

\begin{problem} 
May it be that for all sufficiently large $\Delta$, we have $h_3(\Delta)\le \frac{253}{225}\Delta^3$~?
\end{problem}

\section{Proof of Theorem \ref{thm:bipartite_subgraph}}
\label{section:clique_number_line_square}

We first discuss a technical lemma, which is required later. 

For a given $\beta \in [0,1]$, consider the polynomial
\[  P_\beta(x)
   :=
   \beta x^4+(2\beta-1)x^3+\beta^2x^2+\beta(2\beta-1)x+\beta^2.
\]
See that 
\[ P_{\beta}(1) = 4\beta^2 + 2\beta -1 < 0\]
whenever $\beta \le \frac{1}{4}$. Furthermore, if $\beta' > \beta\ge \frac{1}{4}$, then 
\[ P_{\beta'}(x) - P_{\beta}(x) = \left(\beta' - \beta\right)\left(x^4 + 2x^3 + (\beta' + \beta)x^2 + (2\beta' + 2\beta -1)x + (\beta'+\beta)\right)>0\]
for any $x\ge 0$. So let
\[ \beta^* := \min\{\beta\in [0,1/4]: P_{\beta}(x)\ge 0 \text{ for all }x\ge 0\}.\]
One could use sophisticated computational tools to determine $\beta^*$; we avoid doing so since our preliminary computations suggest that $\beta^*$ is close to $0.312028$. Instead, we establish the following using elementary means.

\begin{lemma}\label{lemma:beta} We have 
    \[ \frac{1}{4} < \beta^* \le \frac{620}{1987} < 0.3120282 < \frac{5}{16}.\]
\end{lemma}
\begin{proof}
The inequality $\beta^* > \frac{1}{4}$ is clear from the above discussion. Moreover, 
\begin{align*}
P_{\frac{620}{1987}}(x)
& =
\frac{620}{1987}
\left(
    x^2-\frac{747}{1240}x-\frac{211}{1000}
\right)^2 \\
& \quad +\, 
\frac{
    708654527125x^2
    -1203303717450x
    +510806838853
}{
    6119661950000
}.
\end{align*}
The discriminant of the second quadratic term on the right-hand side above is  
\[
(-1203303717450)^2
-
4\cdot 708654527125\cdot 510806838853  \\
=
-2478929365735048000
<
0,
\]
which implies $P_{\frac{620}{1987}}(x)\ge 0$ for all $x\ge 0$. We conclude that $\beta^* \le \frac{620}{1987}$. 
\end{proof}

We will now argue the following. 

\begin{lemma}\label{lemma:bipartite_subgraph}
Let $G$ be a graph and let $H$ be a bipartite subgraph of $G$. If
$E(H)$ is a clique in $L(G)^2$, then
\[
   |E(H)|\le \beta^*\,\sigma_G(H)^2.
\]
\end{lemma}

\begin{proof}
Consider the subgraph $G[V(H)]$ of $G$ induced by the vertices of $H$. It is clear that $E(H)$ is also a clique in $L(G[V(H)])^2$ and $\sigma_{G[V(H)]}(H) \le \sigma_G(H)$. Thus, we can replace $G$ with $G[V(H)]$; equivalently, we can assume that $H$ is a spanning subgraph of $G$. Throughout, let
\[
   \sigma:=\sigma_G(H)\quad \text{and} \quad \Delta:=\Delta(H).
\]
We proceed by induction on $|E(H)|$. If $|E(H)|=1$, then $\sigma\ge 2$, which implies
\[
   |E(H)|=1\le 4\beta^*\le \beta^*\,\sigma^2,
\]
since $\beta^* > \frac{1}{4}$. 

So, assume $|E(H)|\ge 2$. Choose $v\in V(H)$ such that $\deg_H(v)=\Delta$. Let $X\cup Y$ be a bipartition
of $H$ and assume $v\in X$. Define
\[
   A:=N_H(v)\subseteq Y,  \quad  C:= N_G(v)\setminus A, \quad B:=V(H)\setminus N_G[v], \quad  F:=H[B].
\]
Clearly, every edge in $E(H)\setminus E(F)$ is incident with a vertex in $N_G[v]$. Therefore
\[
   |E(H)\setminus E(F)|\le |E_H(A)| + |E_H(C)|,
\]
where $E_H(S)$ denotes the set of $H$-edges incident with $S\subseteq V(H)$.
Therefore, 
\begin{equation}\label{eq:H_bound_1}
   |E(H)|\le |E_H(A)| + |E_H(C)| + |E(F)|. 
\end{equation}
In what follows, we will estimate $|E_H(A)|$ and $|E_H(C)|$ using structural arguments and $|E(F)|$ by induction hypothesis, thus obtaining the desired bound for $|E(H)|$.

First,
\begin{equation}\label{eq:C_bound}
   |E_H(C)|\le \Delta |C| =\Delta(\deg_G(v)-\Delta).
\end{equation}

Next, we estimate $|E_H(A)|$. Note that for each $a\in A$, the edge $va$ lies in $H$, so $\deg_G(v)+\deg_G(a)\le \sigma.$
Thus
\[
   \sum_{a\in A}\deg_G(a)\le \Delta(\sigma-\deg_G(v)).
\]
Define 
\[
  B_X:=B\cap X,\quad B_Y:=B\cap Y, \quad   \Pi:= |E_G(A,B_Y)| + |E_G(A,B_X)\setminus E(H)|,
\]
where $E_G(S,R)$ denotes the edges of $G$ with one endpoint in $S$ and the other in $R$, where $S, R\subseteq V(G)$. The term $\Pi$ counts the $G$-edges incident with $A$ and $B$, which consume $G$-degree at vertices of $A$ but are not counted as $H$-edges incident with $A$. Thus, 
\begin{equation}\label{eq:A_bound}
    |E_H(A)| \le \sum_{a\in A}\deg_G(a)-\Pi \le \Delta(\sigma-\deg_G(v))-\Pi.
\end{equation}
Note that 
\[
   \Pi= \sum_{y\in B_Y} |N_G(y)\cap A| + \sum_{x\in B_X} |(N_G(x)\cap A)\setminus N_H(x)|.
\]
Take an edge $xy\in E(F)$, where $x\in B_X$ and $y\in B_Y$. For every
$a\in A$, the two $H$-edges $va$ and $xy$ are at distance at most
two in $L(G)$ since $E(H)$ forms a clique in $L(G)^2$. Since $x,y\notin N_G[v]$, this can only happen through
an edge $ax$ or an edge $ay$. Therefore,
\[
   A\subseteq N_G(x)\cup N_G(y) \implies   |A|\le |N_G(x)\cap A| + |N_G(y)\cap A|.
\]
Since $|A| = \Delta$, we get 
\[ \Delta \le  |N_G(y)\cap A|  + |(N_G(x)\cap A)\setminus N_H(x)| + |N_H(x)\cap A|. \]
For $x\in B$, define 
\[ h_x : = |N_H(x)\cap A|, \quad f_x: = \deg_{F}(x).\]
Summing over all $xy\in E(F)$ and using the fact that $f_x, f_y\le \Delta$, we get
\begin{align*}
    \Delta |E(F)|
    & \le  \sum_{y\in B_Y}|N_G(y)\cap A|f_y  + \sum_{x\in B_X}|(N_G(x)\cap A)\setminus N_H(x)|f_x + \sum_{x\in B_X}h_xf_x\\
   & \le  \left(\sum_{y\in B_Y}|N_G(y)\cap A|  + \sum_{x\in B_X}|(N_G(x)\cap A)\setminus N_H(x)|\right)\Delta + \sum_{x\in B_X}h_xf_x\\
   &  = \Pi \Delta +  \sum_{x\in B_X}h_xf_x
\end{align*}
for any $xy\in E(F)$, which implies
\begin{equation}\label{eq:Pi_bound}
   |E(F)|\le \Pi+\frac{1}{\Delta}\sum_{x\in B_X}h_xf_x. 
\end{equation}
Combining \eqref{eq:H_bound_1}, \eqref{eq:C_bound}, \eqref{eq:A_bound} and \eqref{eq:Pi_bound}, we get 
\begin{equation}\label{eq:H_bound_2}
   |E(H)| \le \Delta(\sigma-\Delta) + \frac{1}{\Delta}\sum_{x\in B_X}h_xf_x. 
\end{equation}
Note that 
\begin{equation}\label{eq:h_x_estimate}
  \sum_{x\in B_X} h_x  = |E_H(A, B_X)| \le \sum_{a\in A} \deg_H(a)\le |A|\Delta = \Delta^2.  
\end{equation}
Also, $F$ is a bipartite subgraph of $G[B]$ such that $E(F)$ induces a clique in $L(G[B])^2$. Moreover, as discussed before, for any edge $xy\in E(F)$, every vertex in $A$ is adjacent to either $x$ or $y$. This means
\[ \deg_{G[B]}(x) + \deg_{G[B]}(y) \le \deg_G(x) + \deg_{G}(y) - \Delta \le \sigma - \Delta,\]
which implies 
\[ \sigma_{G[B]}(F) \le \sigma - \Delta.\]
Since $|E(F)| < |E(H)|$, by induction hypothesis,
\begin{equation}\label{eq:f_x_estimate}
   \sum_{x\in B_X}f_x = |E(F)|\le \beta^*\ \sigma_{G[B]}(F)^2 \le \beta^* (\sigma - \Delta)^2. 
\end{equation}
Define 
\[ p: = \frac{\beta^*(\sigma-\Delta)^2}{\Delta^2 + \beta^*(\sigma-\Delta)^2} \quad \text{and} \quad q: = 1-p = \frac{\Delta^2}{\Delta^2 +\beta^*(\sigma-\Delta)^2}.\]
For any $x\in B_X$, we have $h_x + f_x \le \Delta$, which implies 
\[
   \frac{h_xf_x}{\Delta}
   \le
   \frac{h_xf_x}{h_x+f_x},
\]
with the convention that the right side is $0$ whenever $h_x + f_x = 0$. Since $p + q = 1$, it is easy to check that 
\[    \frac{h_xf_x}{h_x+f_x} \le p^2 h_x + q^2 f_x.\]
Indeed, 
\[ (p^2 h_x + q^2 f_x)(h_x + f_x) - h_xf_x = (ph_x - qf_x)^2 \ge 0.\]
Summing over $x\in B_X$, we get
\[ \frac{1}{\Delta}\sum_{x\in B_X}h_xf_x \le  p^2 \Delta^2 + q^2 \beta^* (\sigma - \Delta)^2 = \frac{\Delta^2\beta^*(\sigma-\Delta)^2}{\Delta^2 + \beta^*(\sigma-\Delta)^2},\]
by \eqref{eq:h_x_estimate} and \eqref{eq:f_x_estimate}. Hence, by \eqref{eq:H_bound_2}, we have
\[
   |E(H)| \le \Delta(\sigma-\Delta) + \frac{\Delta^2\beta^*(\sigma-\Delta)^2}{\Delta^2 + \beta^*(\sigma-\Delta)^2}.
\]
Set $t=\Delta/\sigma$. Then
\[
   \frac{|E(H)|}{\sigma^2}
   \le
   t(1-t)+
   \frac{\beta^* t^2(1-t)^2}{t^2+\beta^*(1-t)^2}.
\]
Observe that since $|E(H)|>0$, we have $t\in (0,1)$. Now, by Lemma \ref{lemma:beta}, 
\[P_{\beta^*}\left(\frac{t}{1-t}\right)\ge 0\] for all $t<1$. After simplification, one can see that this inequality is equivalent to 
\[t(1-t)+
   \frac{\beta^* t^2(1-t)^2}{t^2+\beta^*(1-t)^2}\le \beta^*. 
\]
Therefore,
\[
   |E(H)|\le \beta^*\sigma^2,
\]
as required.
\end{proof}

Theorem \ref{thm:bipartite_subgraph} now follows from Lemmas \ref{lemma:beta} and \ref{lemma:bipartite_subgraph}.  

\section{Counterexamples to some conjectures on $h_3(\Delta)$}
\label{section:counterexample_power_3}

\subsection{Small counterexamples}

Let $O_4$ be the Odd graph\footnote{See \href{https://mathworld.wolfram.com/OddGraph.html}{Odd Graph} at Wolfram MathWorld} $KG(7,3)$, i.e., $V(O_4)$ is the set of $3$-subsets of $\{1, \ldots, 7\}$ and two vertices are adjacent precisely when they are disjoint. It is known that $O_4$ is a distance regular graph of order $35$, degree $4$, size $70$ and diameter $3$ (cf. \cite{Brouwer_Cohen_Neumaier_1989}). We observe the following.

\begin{lemma}\label{lemma:diam_3_O_4} We have $\diam(L(O_4))\le 3$. Equivalently, $L(O_4)^3$ is a complete graph.
\end{lemma}

\begin{proof}
Let $AB$ and $CD$ be two distinct edges of $O_4$. By the definition of $O_4$, the sets $A$ and $B$ are disjoint $3$-subsets of $\{1, \ldots, 7\}$, and so $|A\cup B|=6$. Similarly, $|C\cup D|=6$. Hence
\[
    |(A\cup B)\cap (C\cup D)|\ge 5.
\]
But the four sets
\[
    A\cap C,
    \quad A\cap D ,
    \quad B\cap C,
    \quad B\cap D
\]
partition $(A\cup B)\cap (C\cup D)$. Therefore, at least one of them has size at
least $2$. This gives endpoints $U\in\{A,B\}$ and $V\in\{C,D\}$ with
$|U\cap V|\ge 2$.

If $|U\cap V|=3$, then $U=V$ implying that $AB$ and $CD$ are incident. If $|U\cap V|=2$, then $U\cup V$ has size $4$,
so $\{1,\ldots, 7\}\setminus (U\cup V)$ is a $3$-set disjoint from both $U$ and $V$. This
$3$-set is a common neighbour of $U$ and $V$ in $O_4$ implying that the distance between $AB$ and $CD$ is at most three in $L(O_4)$. This completes the proof. 
\end{proof}

By the above Lemma \ref{lemma:diam_3_O_4}, it follows that 
\[ h_3(4)\ge |E(O_4)|+1 = 71 >  4^3-4^2+4+2=54 \]
Thus, Conjecture \ref{conj:power_3} is false for $\Delta = 4$.

Now, consider the Steiner system $S(5,8,24)$, i.e., this is the Steiner system with a point set $\Omega$ of $24$ elements, a collection $\mathcal{O}$ of $8$-subsets of $\Omega$ called \emph{octads} such that any $5$-subset of $\Omega$ is contained in exactly one octad in $\mathcal{O}$. The \emph{(large) Witt graph}\footnote{See \href{https://mathworld.wolfram.com/LargeWittGraph.html}{Large Witt Graph} at Wolfram MathWorld} has vertex set $\mathcal{O}$ and two vertices $A, B\in \mathcal{O}$ are adjacent if and only if $A\cap B = \emptyset$. 

Now, fix an element $\infty \in \Omega$. Let $W$ denote the subgraph of the Witt graph induced by the octads of $\mathcal{O}$ not containing $\infty$. The graph $W$ is known as the \emph{truncated Witt graph} \footnote{See \href{https://mathworld.wolfram.com/TruncatedWittGraph.html}{Truncated Witt Graph} at Wolfram MathWorld} (cf. \cite{Brouwer_Cohen_Neumaier_1989}).

It is known that $W$ is a distance regular graph of order $506$, degree $15$, size $3795$ and diameter $3$. We observe the following. 

\begin{lemma}\label{lemma:diam_W} We have $\diam(L(W))\le 3$, i.e., $L(W)^3$ is a complete graph.
\end{lemma}

\begin{proof}
Let $AB, CD\in E(W)$ be two distinct edges of $W$.  Thus
\[
        A\cap B=\emptyset
        \qquad\text{and}\qquad
        C\cap D=\emptyset.
\]
The sets $A\cup B$ and $C\cup D$ are both $16$-subsets of
$\Omega\backslash \{\infty\}$. Hence,
\[
|(A\cup B)\cap(C\cup D)|\ge 16+16-23 = 9.
\]
The four sets
\[
        A\cap C,\quad A\cap D,\quad B\cap C,\quad B\cap D
\]
are pairwise disjoint and partition $(A\cup B)\cap(C\cup D)$. It is known that two octads meet in $0,2,4$, or $8$ points. Therefore, at least one of these four
intersections has size $4$ or $8$. Thus, there exist endpoints $U\in\{A,B\}$ and $V\in\{C,D\}$ such that either $U=V$, or $|U\cap V|=4$. 

If $U=V$, then the two edges $AB$ and $CD$ are incident in $W$, and
so their distance in $L(W)$ is at most $1$. So, suppose $|U\cap V|=4$.  We claim that $U$ and $V$ have a common
neighbour in $W$.  

Note that the \emph{tetrad} $U\cap V$ determines a \emph{sextet}, i.e., a partition
\[
        \Omega=T_0\cup T_1\cup T_2\cup T_3\cup T_4\cup T_5
\]
into six tetrads such that the union of any two tetrads
of the sextet is an octad with $T_0=U\cap V$ and
\[
        U=T_0\cup T_1,
        \qquad
        V=T_0\cup T_2.
\]
Since $U,V\subseteq \Omega\backslash \{\infty\}$, we can assume without loss of generality that $\infty\in T_5$. Then, $T_3\cup T_4$ is an octad of $\Omega\backslash \{\infty\}$ such that
\[
        (T_3\cup T_4)\cap U=(T_3\cup T_4)\cap V=\emptyset.
\]
In other words, $T_3\cup T_4$ is a vertex of $W$ adjacent to both $U$ and $V$ in $W$. 

Thus, for any two edges of $W$, some endpoint of one is at distance at most $2$ in $W$ from some endpoint of the other. The proof is complete.
\end{proof}

In light of the above Lemma \ref{lemma:diam_W}, we see that
\[h_3(15)\ge |E(W)|+ 1 = 3796 > 15^3-15^2+15+2 = 3167.\]
Therefore, $W$ is a counterexample to Conjecture \ref{conj:power_3} for $\Delta = 15$. 

\subsection{Infinite family of counterexamples}

We now construct an infinite family of counterexamples.

Let $q$ be a prime power and consider the finite field $\mathbb{F}_q$. For a vector $x = (x_1,x_2,x_3)\in \mathbb{F}_q^3$, we denote by $\langle x\rangle$ the $1$-dimensional subspace of $\mathbb{F}_q^3$ generated by $x$. Then $\PG(2,q)$ denotes the set of all $1$-dimensional subspaces of $\mathbb{F}_q^3$. The elements of $\PG(2,q)$ are called \emph{projective points}. It is well-known that 
\[|\PG(2,q)|=q^2+q+1.\]

Consider the standard inner product on $\mathbb{F}_q^3$ given by
\[
        \langle x,y\rangle : =x_1y_1+x_2y_2+x_3y_3,
\]
for every $x = (x_1, x_2, x_3), y = (y_1, y_2, y_3) \in \mathbb{F}_q^3$.

For a projective point $\alpha = \langle x \rangle \in \PG(2,q)$, define 
\[
        \alpha^\perp:=\{\langle y\rangle\in \PG(2,q):\langle x,y\rangle=0\}.
\]
Then $\alpha^\perp$ is a \emph{projective line} and hence
\[
        |\alpha^\perp|=q+1.
\]
Moreover, for any $\alpha, \beta\in \PG(2,q)$, the two projective lines $\alpha^\perp$ and $\beta^\perp$ meet, i.e.,
\[
        \alpha^\perp\cap \beta^\perp\ne \emptyset.
\]

Now, for a given simple graph $H$, define a simple graph $G[H,q]$ as follows. Put
\[
        V(G[H,q]):=V(H)\times \PG(2,q),
\]
and two vertices $(u,\alpha)$ and $(v,\beta)$ are adjacent in $G[H,q]$ if and only if
\[
        uv \in E(H)
        \quad\text{and}\quad
        \beta\in \alpha^\perp.
\]
The adjacency relation is symmetric because the bilinear form is symmetric:
$\beta\in \alpha^\perp$ if and only if $\alpha\in \beta^\perp$.
Equivalently, the adjacency matrix of $G[H,q]$ is given by 
\[ A(G[H,q]) = A(H)\otimes M,\]
where $A(H)$ is the adjacency matrix of $H$, $\otimes$ denotes the Kronecker product, and $M$ denotes the adjacency matrix of the \emph{looped polarity graph} of $\PG(2,q)$, i.e.,
\[ M_{\alpha, \beta} = 1 \iff \beta \in \alpha^\perp.\]
It is possible for $\alpha \in \PG(2,q)$ to lie in $\alpha^\perp$ and so $M$ can have non-zero diagonal entries. Refer \cite{Cameron_1992, Bachraty_Siran_2015} for further details on projective planes and polarity graphs.

\begin{lemma}\label{lemma:G[H,q]} Let $H$ be a $\Delta(H)$-regular graph with $\diam(L(H))\le 3$. For every prime power $q$, the following properties hold for the graph $G[H,q]$:
\begin{enumerate}[$(i)$]
    \item The graph $G[H,q]$ is regular with $\Delta(G[H,q]) =\Delta(H) (q+1)$.
    \item $|E(G[H,q])|=|E(H)|(q+1)(q^2+q+1)$.
    \item $\diam(L(G[H,q]))\le 3.$
\end{enumerate}
\end{lemma}

\begin{proof}
Consider a vertex $(u,\alpha)\in V(G[H,q])$. The vertex $u$ has degree $\Delta(H)$ in $H$. For each neighbour $v$ of $u$ in $H$, the possible second coordinates $\beta$ are precisely the points of the line $\alpha^\perp$, of which there are $q+1$. Therefore, every vertex $(u,\alpha)$ of $G[H,q]$ has degree $\Delta(H)(q+1).$ This proves $(i)$.

To see $(ii)$, observe that 
\begin{align*}
|E(G[H,q])| &= \frac{1}{2}\cdot |V(G[H,q])|\cdot \Delta(G[H,q]) \\
& = \frac{1}{2} \cdot |V(H)|\cdot |\PG(2,q)| \cdot  \Delta(H) \cdot (q+1)\\
& = |E(H)| \cdot (q+1)\cdot (q^2 + q + 1).  
\end{align*}

We now prove $(iii)$. Let
\[
        e=\{(u,a),(v,b)\}
        \quad \text{and}\quad 
        f=\{(z,c),(w,d)\}
\]
be two distinct edges of $G[H,q]$. Then $uv$ and $zw$ are edges of $H$. Since $\diam(L(H))\le 3$, there exist endpoints $r\in\{u,v\}$ and
$s\in\{z,w\}$ such that either $r=s$ or there is a common neighbour of $r$ and $s$ in $H$. 
 
Let $\alpha$ denote the projective point such that $(r,\alpha)$ is an endpoint of $e$, and let $\beta$ denote the projective point such that $(s, \beta)$ is an endpoint of $f$. Since the two projective lines $\alpha^\perp$ and $\beta^\perp$
meet, choose
\[
        \gamma\in \alpha^\perp\cap \beta^\perp.
\]

If $r=s$, choose $p$ to be any neighbour of $r$ in $H$. If $r\neq s$, then choose $p$ to be the common neighbour of $r$ and $s$ in $H$. In either case, 
\[
        (r,\alpha)\sim (p,\gamma)\sim (s,\beta).
\]
Thus, some endpoint of $e$ is at distance at most $2$ from some endpoint of $f$ in $G[H,q]$. Consequently, the two edges $e$ and $f$ have distance at most $3$ in
the line graph $L(G[H,q])$. Since $e$ and $f$ were arbitrary, we conclude that $\diam L(G[H,q])\le 3$.
\end{proof}

\begin{lemma}\label{lemma:h_3_lower_bound}
    Let $H$ be a $\Delta(H)$-regular graph with $\diam(L(H))\le 3$. Then 
    \[ \liminf_{\Delta \to \infty} \frac{h_3(\Delta)}{\Delta^3}\ge \frac{|E(H)|}{\Delta(H)^3}.\]
\end{lemma}

\begin{proof}
Assume $\Delta$ to be arbitrary and sufficiently large. By the Prime Number Theorem, for any $\varepsilon > 0$ and any sufficiently large number $N$, there exists a prime $p\in [(1-\varepsilon)N, N]$. Take $N = \frac{\Delta}{\Delta(H)}-1$ and choose a prime $p$ so that
\[ \Delta(1-o_{\Delta}(1)) \le \Delta(H)(p+1) \le \Delta. \]
Then, using Lemma \ref{lemma:G[H,q]}, we see that
\begin{align*}
    \frac{h_3(\Delta)}{\Delta^3} & \ge \frac{h_3(\Delta(H)(p+1))}{\Delta^3}\\
    &\ge \frac{|E(G[H,p])|}{\Delta(H)^3(p+1)^3}\cdot \frac{\Delta(H)^3(p+1)^3}{\Delta^3} \\
    & = \frac{|E(H)|(p+1)(p^2+p+1)}{\Delta(H)^3(p+1)^3}\cdot (1-o_\Delta(1))^3 \\
    & = \frac{|E(H)|}{\Delta(H)^3}(1-o_p(1)).
\end{align*}
The assertion follows. 
\end{proof}

We can now take $H$ to be either the odd graph $O_4$ or the truncated Witt graph $W$ to construct the desired counterexamples. 

Taking $H = O_4$ and using Lemmas \ref{lemma:diam_3_O_4} and \ref{lemma:h_3_lower_bound}, we have 
\[ \liminf_{\Delta\to\infty}\frac{h_3(\Delta)}{\Delta^3}\ge \frac{35}{32}.\]

Again, taking $H=W$, and using Lemmas \ref{lemma:diam_W} and \ref{lemma:h_3_lower_bound}, we get 
\[
\liminf_{\Delta\to\infty}\frac{h_3(\Delta)}{\Delta^3}\ge \frac{253}{225}.
\]
Since $\frac{253}{225}>\frac{35}{32}$, the truth of Theorem \ref{thm:infinite_counter_power_3} is clear.

\section*{Acknowledgements}
Bojan Mohar is supported in part by the NSERC Discovery Grant R832714 (Canada), by the ERC Synergy grant (European Union, ERC, KARST, project number 101071836), and by the Research Project N1-0218 of ARIS (Slovenia). The authors thank Aida Abiad for inspiring our research on graph powers.

\section*{AI statement}
We acknowledge the use of AI tools during the ideation phase. We declare that the text is not AI-generated.

\bibliographystyle{plain}
\bibliography{references}

\vspace{0.4cm}

\affl{Hitesh Kumar}{hitesh.kumar.math@gmail.com, hitesh\_kumar@sfu.ca}{Department of Mathematics, Simon Fraser University, Burnaby, Canada}

\affl{Bojan Mohar}{mohar@sfu.ca}{Department of Mathematics, Simon Fraser University, Burnaby, Canada\\On leave from FMF, Department of Mathematics, University of Ljubljana.}
 
\affl{Shivaramakrishna Pragada}{shivaramakrishna\_pragada@sfu.ca, shivaramkratos@gmail.com}{Department of Mathematics, Simon Fraser University, Burnaby, Canada}

\end{document}